\documentclass{IEEEtran}

\usepackage{preamble}
\usepackage{comment}
\usepackage{setspace}
\usepackage{url}

\usepackage[
colorlinks=true,
linkcolor=blue,
citecolor=blue,
urlcolor=blue
]{hyperref}

\def\BibTeX{{\rm B\kern-.05em{\sc i\kern-.025em b}\kern-.08em
    T\kern-.1667em\lower.7ex\hbox{E}\kern-.125emX}}
\begin{document}
\title{On the Value Function of Infinite-Horizon Optimal Control of Piecewise Affine Systems}
\author{Francesco Cordiano*\thanks{*Equal contribution. This project has received funding from the European Research Council (ERC) under the European Union's Horizon 2020 research and innovation programme (Grant agreement No.\ 101018826 - ERC Advanced Grant CLariNet), and by the Rubicon Postdoctoral Fellowship (Correspondence No. 2026/ENW/02250137), funded by the Netherlands Organisation for Scientific Research (NWO).

Francesco Cordiano and Bart De Schutter are affiliated with the Delft Center for Systems and Control, Delft University of Technology, 2628 CD, Delft, The Netherlands, email: $\{$f.cordiano, b.deschutter$\}$@tudelft.nl. Kanghui He is affiliated with the Department of Engineering Science, University of Oxford, OX1 3PJ Oxford, U.K., email: kanghui.he@eng.ox.ac.uk.
}, Kanghui He*, and Bart De Schutter, \IEEEmembership{Fellow, IEEE}}

\maketitle

\begin{abstract}
In this paper, we study the structure of the value function in constrained infinite-time optimal control (CITOC) problems of piecewise affine (PWA) systems, with $\ell_1$ or $\ell_\infty$ stage cost.
Existing works, such as \cite{baotic2006constrained}, establish that the resulting value function is PWA in the state. However, existing results do not analyze whether the value function is a proper PWA function, i.e., with a finite number of affine pieces over compact sets, or whether the number of pieces can be infinite. We show that the latter case is indeed possible by means of an explicit example, which is also instrumental in establishing rigorous and easily verifiable sufficient conditions that ensure that the resulting value function is a proper PWA function.
Our theoretical findings complement well-known results, e.g., the linear-quadratic case, and serve as support for recent learning-based control schemes for PWA systems. Throughout the paper, the proposed results are illustrated by means of a numerical example.
\end{abstract}


\section{Introduction}
\label{sec:introduction}
Over the last decades, significant interest has been devoted to the analysis and control of piecewise affine (PWA) systems \cite{bemporad1999control}, which are a special class of hybrid systems \cite{van2007introduction} where the function describing the dynamics is piecewise affine in the state and the input. Their popularity stems from the ability to model dynamical systems with hybrid dynamics (e.g., energy storage systems \cite{arnold2009modelbased}), or to approximate nonlinear systems with a user-defined level of accuracy \cite{gharavi2023efficient}. In addition, several classes of hybrid systems, including mixed-logical dynamical, max-min-plus-scaling, and linear complementarity systems, have been proved to be equivalent to PWA systems (see \cite{heemels2001equivalence} for an overview of the aforementioned models and their equivalence).

A variety of strategies have been studied for optimal control of constrained PWA systems, ranging from dynamic programming (DP) \cite{baotic2006constrained}, to model predictive control \cite{deschutter2004mpc, lazar2006stabilizing}, and learning-based methods \cite{he2024approximatea_pwa}.
In particular, recent approaches leveraging learning algorithms \cite{he2024approximatea_lin, he2024approximatea_pwa}, rely on the structure of the value function to ensure adequate quality of the approximation. For example, it is known from \cite{bemporad2002explicit} that the value function of a constrained LQR problem is piecewise quadratic, so \cite{he2024approximatea_lin} proposes piecewise quadratic neural networks to approximate the corresponding value function. Similarly, \cite{baotic2006constrained} states that the value function of a constrained optimal control problem for PWA systems with $\ell_1$ or $\ell_\infty$ stage cost has a PWA structure; therefore, \cite{he2024approximatea_pwa} employs a PWA neural network to approximate it.

Due to the importance of these results in recent learning-based frameworks, this technical note presents a detailed analysis of the value function for constrained optimal control of PWA systems with $\ell_1$ or $\ell_\infty$ stage cost. Although the value function is known to be PWA {with a finite number of pieces} for finite-time optimal control problems, we show that this is not necessarily true for the infinite-horizon case, unless additional sufficient conditions are specified. Our contributions are summarized as follows:
\begin{itemize}
    \item {We consider} constrained infinite time optimal control (CITOC) of PWA systems with $\ell_1$ or $\ell_\infty$ stage cost. We construct an example for which we can calculate the resulting value function in closed form, showing that {it has infinitely many affine pieces on a compact set.}
    \item {The proposed example unveils that this structural property of the value function is related to the controllability of the system and to the design of the cost matrices. Hence, based on these insights, we derive explicit and easily verifiable sufficient conditions such that the value function of a CITOC problem for a PWA system with an $\ell_1$ or $\ell_\infty$ stage cost has a proper PWA structure, i.e., with a finite number of affine pieces over compact sets.} 
\end{itemize}
This study complements related papers (e.g., the well-known linear-quadratic case \cite{bemporad2002explicit}) and also covers the case of linear systems  with $\ell_1$ or $\ell_\infty$ stage cost. Moreover, it serves as support for {learning-based approaches \cite{he2024approximatea_lin, he2024approximatea_pwa} that} aim to design function approximators for CITOC problems, while providing guarantees on the quality of the approximation.

The paper is organized as follows: Section \ref{sec:problem} recaps the problem setting in \cite{baotic2006constrained}, Section \ref{sec:counterexample} proposes {an explicit example in which the value function has infinitely many pieces}, in Section \ref{sec:suff} we define sufficient conditions {such that this does not occur}, and Section \ref{sec:conclusions} concludes the paper. 
Our theoretical findings are accompanied by a numerical example throughout the paper.

\section{Problem statement}\label{sec:problem}
We consider the class of discrete-time piecewise affine (PWA) systems of the following form:
\begin{align}\label{eq:pwa_sys}
\begin{split}
    x_{t+1} &= \fpwa(x_t, u_t)
    \\& = A_i x_t + B_i u_t + f_i, \quad \text{if }
    [x_t^\top\ u_t^\top]^\top
    \in \Dcal_i,
\end{split}
\end{align}
where, {for all $t\in\Z_{\geq0}$},  $x_t{\in\R^n}$ is the state, $u_t{\in\R^m}$ is the control input, and the domain $\Dcal = \cup_{i=1}^{n_d} \Dcal_i$ is a {nonempty subset of} $\R^{n+m}$. In particular, {$\{\Dcal_i\}_{i=1}^{n_d}$} constitutes a {non-overlapping} polyhedral partition of $\Dcal$, {i.e., for all $i\in\{1,...,n_d\}$, we have
\begin{align*}
\Dcal_i = \{[x^\top \ u^\top] \in\R^{n+m}:\  & D_i^x x + D_i^u u \leq D^0_i, 
\\& \bar D_i^x x + \bar D_i^u u < \bar D^0_i\},
\end{align*}
with} $\Dcal_i\neq\emptyset, \forall i\in\{1,...,n_d\}$, and {$\Dcal_i\cap \Dcal_j = \emptyset, \forall i\neq j$}. {For our contributions, we do not necessarily require that the dynamics in \eqref{eq:pwa_sys} are continuous. Continuity of the system dynamics is generally necessary for the existence of a minimizer policy over $\Dcal$ as in \cite{baotic2006constrained}, but it is not strictly required for the analysis proposed in this paper.}

Let us consider the following cost function
$$J_\infty(x_0, U_\infty) := \lim_{T\rightarrow\infty} \sum_{t=0}^T g(x_t, u_t),$$
where $U_\infty$ denotes the input sequence $\{u_t\}_{t=0}^\infty$, and $g(x,u) := \|Qx\|_p + \|Ru\|_p$, with $p\in\{1,\infty\}$, {and $Q\in\R^{n_Q\times n}$, $R\in\R^{m_R\times m}$}.
We then consider the following constrained infinite-time optimal control problem (CITOC) 
\begin{align}
    J^\star_\infty(x_0) = \min_{U_\infty} 
    & \ J_\infty(x_0, U_\infty) \label{citoc:cost}
    \\ \text{s.t.} & \ x_{t+1} = \fpwa(x_t, u_t), \ {\forall t\in\Z_{\geq0} } \label{citoc:dyn}
    \\ & \ [x_t^\top \ u_t^\top]^\top\in\Dcal, \ \forall t\in\Z_{\geq0} \label{citoc:constr},
\end{align}
with minimizer $U^\star_\infty := \{{u^\star_t}\}_{t=0}^\infty$.
For CITOC \eqref{citoc:cost}--\eqref{citoc:constr} to be well-posed, \cite{baotic2006constrained} considers the following assumptions:
\begin{assumption}\label{assum:origin}
    The system \eqref{eq:pwa_sys} is stabilizable. In addition, the origin in the extended state-input space is an equilibrium point of the PWA system \eqref{eq:pwa_sys}, i.e., $0_{n+m}\in\Dcal$ and $0_n = \fpwa(0_n, 0_m)$, where $0_n = [0 \ ... \ 0]^\top\in\R^n$.
\end{assumption}
\begin{assumption}\label{assum:Jfinite}
    The CITOC problem \eqref{citoc:cost}--\eqref{citoc:constr} is well defined, i.e., the minimum is achieved for some feasible input sequence $U^\star_\infty$, and $J^\star_\infty < \infty$ for any feasible state $x$ on {a set $\Xcal_\infty:=\{x:J^\star_\infty(x)<\infty\}$}. In addition, the matrix $Q$ has full column rank, i.e., $Qx=0$ if and only if $x=0_{n}$.
\end{assumption}

The authors of \cite{baotic2006constrained} propose to solve the CITOC \eqref{citoc:cost}--\eqref{citoc:constr} via the following {dynamic programming (DP)} iterations. In particular, let us consider the following initialization:
\begin{align}\label{dp:iter_0}
    J_0(x) = 0, \, \Xcal_0 = \left\{x\in\R^n \mid \exists u\in\R^{m}: \begin{bmatrix}
        x\\u
    \end{bmatrix}\in\Dcal \right\}.
\end{align}
Then, the DP updates consist of
\begin{align}
J_{k}(x) &= \begin{aligned}[t]\label{dp:iter_1}
        \min_{u} & \ g(x, u) + J_{k-1}(\fpwa(x, u))
        \\ \text{s.t.} & \ \fpwa(x, u) \in \Xcal_{k-1}
        \\& (x,u)\in\Dcal,
    \end{aligned}
    \\ \Xcal_k & = \begin{aligned}[t]\label{dp:iter_2}
        {\Big\{ x\in\R^n \mid \exists u\in\R^m :\; } & { \begin{bmatrix}
    x\\u
    \end{bmatrix} \in \Dcal,}  
    \\ & {\fpwa(x,u) \in \Xcal_{k-1} \Big\}} ,
    \end{aligned} 
\end{align}
and we denote by $\mu_k$ the policy resulting from the minimization problem in \eqref{dp:iter_1}, with $\mu_0=0$.
Under Assumptions \ref{assum:origin} and \ref{assum:Jfinite}, in \cite{baotic2006constrained} it is shown that the solution of the CITOC \eqref{citoc:cost}--\eqref{citoc:constr} and the solution of the DP iterations \eqref{dp:iter_0}--\eqref{dp:iter_2} coincide, and that the closed-loop system \eqref{eq:pwa_sys} is asymptotically stable under the control sequence ${U^\star_\infty}$, i.e., $\lim_{t\rightarrow\infty} x_t = 0$.

In \cite{baotic2006constrained}, Theorem IV.9 states that, under Assumption \ref{assum:Jfinite}, the solution of the CITOC \eqref{citoc:cost}--\eqref{citoc:constr} is a PWA function, i.e.:
\begin{align*}
    J^\star_\infty(x) = \Phi_{\infty,i} x + \Gamma_{\infty, i} \quad \text{if } x\in\Pcal_{\infty,i},
\end{align*}
and the optimal state feedback control law is also PWA, i.e.:
\begin{align*}
    \mu_{\infty}^\star(x) = F_{\infty,i}x + G_{\infty,i} \quad \text{if } x\in\Pcal_{\infty,i},
\end{align*}
where $\{\Pcal_{\infty,i}\}_{i=1}^{N_\text{pwa}}$ is a polyhedral partition of the set $\Xcal_\infty$.
However, in this paper, we demonstrate that {$N_\text{pwa}=\infty$ can actually occur, and additional sufficient conditions are needed for $J^\star_\infty$ to have a finite number of pieces over a compact set.} First, we give the definition of a PWA function that is commonly accepted in the literature \cite{huchette2023nonconvex, gorokhovik1994piecewise}:
\begin{definition}\label{def:pwa}
    {Let $\Scal\subset\R^{n_1}$ be a polyhedral set, and consider $f:\Scal\rightarrow\R^{n_2}$, for some $n_1, n_2\in\Z_{>0}$.} The function $f$ is PWA if: {\emph{i)} $f(x) = E_i x + d_i$} if $x\in\Scal_i$, {for some $E_i\in\R^{n_2\times n_1}$ and $d_i\in\R^{n_2}$}, where $\{\Scal_i\}_{i=1}^{N_\text{pwa}}, {N_\text{pwa}\in\Z_{>0}}$, is a polyhedral partition of $\Scal$; {\emph{ii)}} ${N_\text{pwa}}<\infty$.
\end{definition}

We emphasize that, according to Definition \ref{def:pwa}, a PWA function needs a \emph{finite} number of affine pieces over a compact set. 
{It is well-known that,
in the DP iterations \eqref{dp:iter_0}--\eqref{dp:iter_2}, the number of pieces of $J_k$ is finite for any finite $k$ \cite{baotic2006constrained}, and it can increase rapidly with the iterations \cite{borrelli2003constrained}. However, the set of PWA functions is not closed, i.e., it is possible to find sequences of PWA functions
that converge to a function that is not PWA according to Definition \ref{def:pwa}. This means that $J^\star_\infty$ may not be PWA. As an example, consider the function $G$ defined by $G(x)=x^2, \forall x\in[0, 1]$, which is a smooth function. Consider also $G_z$, i.e., its PWA approximation with $z\in\Z_{>0}$ pieces, obtained by interpolating the values $G(\bar x)$ for each $\bar x \in \{0, \frac{1}{z}, \frac{2}{z}, ..., 1\}$. For $x\in[\frac{i}{z}, \frac{i+1}{z}]$, with $i\in\{0,...,z-1\}$, the analytic expression for $G_z$ is $G_z(x) = \frac{2i+1}{z}x - \frac{i(i+1)}{z^2}$. Therefore, for $x\in[\frac{i}{z}, \frac{i+1}{z}]$, we have
\begin{align*}
    |G_z(x) - x^2| 
    & = \left|x-\frac{i}{z}\right|\left|\frac{i+1}{z}-x \right| 
    \\& \leq \frac{1}{z}\cdot \frac{1}{z} \underset{z\to\infty}{\longrightarrow} 0.
\end{align*}
Hence, $G_z$ converges uniformly to $G(x)=x^2$, which is not a PWA function according to Definition 1.
}

{More generally}, two cases may occur: $i$) the resulting function is smooth, or $ii$) the resulting function has infinitely many affine pieces in a compact set, while still being non-smooth over that set. In Section \ref{sec:counterexample}, we show, by means of a counterexample, that case $ii)$ can actually occur, even for a very simple class of PWA systems. Moreover, in Section \ref{sec:suff} we state which additional conditions are needed {for $J^\star_\infty$ to be PWA in the standard sense.}

\section{Explicit example}\label{sec:counterexample}
We consider a class of 2-dimensional linear systems such as:
\begin{align}\label{eq:lin_sys}
    A = \begin{bmatrix}
        a_1 & a_2 \\ 0 & a_1
    \end{bmatrix}
    \quad B = \begin{bmatrix}
        1 & 0\\ 0 & 1
    \end{bmatrix}
\end{align}
with $a_1, a_2\in\R$, $a_1 \in (-1,1), a_1\neq 0$. In principle, a counterexample can be found for a more general class of systems. However, the scope of this section is to show that the result of Theorem IV.9 in \cite{baotic2006constrained} {allows an infinite number of pieces} even for simple systems that, like the ones in \eqref{eq:lin_sys}, enjoy desirable properties, such as asymptotic stability and controllability.

The proposed {system class} is sufficient {for this purpose} for two reasons: 1) The trivial reason is that linear systems are a special class of PWA systems; 2) If the origin is in the interior of a region of the partition of the PWA dynamics, then the system dynamics will be linear in a sufficiently small set around the origin, since the $A$ matrix is Schur-stable.

For this {example}, we define the following cost function:
\begin{align}\label{eq:count_cost}
    g(x,u) = |x_1| + |x_2| + R_1|u_1| + R_2|u_2|,
\end{align}
where $R_1, R_2\in\R_{>0}$ will be defined later.
We want to solve the infinite-horizon problem {with constraint $x_t\in\Xcal, \forall t\in\Z_{\geq0}$}, where $\Xcal$ is a compact set containing the origin {($\Dcal$ in the formulation \eqref{citoc:cost}--\eqref{citoc:constr} is simply replaced by $\Xcal$ in this example)}. Thus, {for a given $x_0\in\Xcal$, we consider}:
\begin{align}\label{citoc_lin}
\begin{split}
    J^\star_{\infty}(x_0) = \min_{U_\infty} & \ \sum_{t=0}^\infty g(x_t, u_t)
    \\ \text{s.t.}& \ x_{t+1} = Ax_t + Bu_t, \ {\forall t\in\Z_{\geq0}}
    \\ & \ {x_t\in\Xcal, \ \forall t\in\Z_{\geq0}},
\end{split}
\end{align}
with $A, B$ chosen as in \eqref{eq:lin_sys}. Since \eqref{citoc_lin} is of the form \eqref{citoc:cost}--\eqref{citoc:constr}, 
we can solve \eqref{citoc_lin} using the DP iterations \eqref{dp:iter_0}--\eqref{dp:iter_2}.

In the following, we provide an expression for $J^\star_\infty$ in \eqref{citoc_lin} in closed form. First, define the coefficients $c_k, d_k, r_{1,k}, r_{2,k} \in \R$, for $k\in\Z_{\geq0}$, with
\begin{align}\label{eq:coeff_init}
    c_0 = 1, \quad d_0 = 0, \quad r_{1,0} = 1, \quad r_{2,0}=1,
\end{align}
with the update rule
\begin{align}\label{eq:coeff_updates}
    \begin{cases}
        c_{k+1} = a_1 c_k
        \\d_{k+1} = a_2c_k + a_1 d_k
        \\r_{1,k+1} = r_{1,k} + |c_k|,
        \\r_{2,k+1} = r_{2,k} + |d_k| + |c_k|,
    \end{cases}
\end{align}
where $a_1$ and $a_2$ are the coefficients of the matrix $A$ defined in \eqref{eq:lin_sys}. Also, define
$$ r_{i,\infty} :=  \lim_{k\rightarrow\infty} r_{i,k}, \quad \forall i\in\{1,2\}.$$
Then, we can prove the following proposition:
\begin{proposition}\label{prop:closed_form}
{Assume that $\Xcal$ is an invariant set for the autonomous system associated to \eqref{eq:lin_sys}, i.e., $Ax\in\Xcal, \forall x\in\Xcal$}. Then, for the system described by \eqref{eq:lin_sys}-\eqref{eq:count_cost}, it holds that:
\begin{enumerate}
    \item\label{item:r_finite} $r_{i,\infty }< \infty, \forall i\in\{1,2\}$;
    \item\label{item:value} If $R_i \geq r_{i,\infty}, \forall i\in\{1,2\}$, the value function at iteration $k$ can be expressed as $J_{k}(x) = \sum_{j=0}^{k-1} (|c_j x_1 + d_j x_2| + |c_j x_2|)$, $\forall x\in\Xcal, \forall k\in\Z_{\geq 1}$, and ${J^\star_\infty}(x)=\lim_{k\rightarrow\infty}J_k(x)$. Also, the optimal policy at iteration $k$ satisfies $\mu_k(x) = 0, \forall x\in\Xcal, \forall k\in\Z_{\geq 1}$.
\end{enumerate}
\end{proposition}
The proof is given in Appendix \ref{app:proof}, and in Appendix \ref{app:explicit}, we give an analytic expression for ${J^\star_\infty}$ to show that it is not {formally} PWA.

The intuition of the above result is that, if the input weights $R_1$ and $R_2$ are large enough, the minimizer of the control problem is $\mu_\infty^\star(x)=0$ for any $x\in\Xcal$, which allows us to easily find a closed-form solution for the optimal value function.

We give an explicit relation for the input weights $R_i$, which follows from $R_{i} \geq r_{i,\infty}, \forall i\in\{1,2\}.$
First, note that we have a simple closed-form solution for \eqref{eq:coeff_updates}, given {the} initial conditions \eqref{eq:coeff_init}:
\begin{align}\label{eq:coeff_convergence}
    \begin{cases}
        c_{k} = a_1^{k}
        \\d_{k} = k a_1^{k-1} a_2
        \\r_{1,k} = 1 + \sum_{l=0}^{k-1}|a_1^l| 
        \\r_{2,k} = 1 + {|a_2|} \sum_{l=0}^{k-1} l|a_1^{l-1}| +  \sum_{l=0}^{k-1}|a_1^l| 
    \end{cases}
\end{align}
from which we have
\begin{align*}
    &r_{1,\infty} = \lim_{k\rightarrow\infty} r_{1,k} = \frac{1}{1-|a_1|}
    \\& r_{2,\infty} = \lim_{k\rightarrow\infty} r_{2,k} = 1 + \frac{|a_2|}{(1-|a_1|)^2} + \frac{1}{1-|a_1|}.
\end{align*}
By setting, e.g., $a_1 = 0.8, a_2 = 1$, {we observe that $\Xcal=[-10, 10]\times[-2, 2]$ is invariant for the corresponding autonomous system. Then,} we find that $R_1 = r_{1,\infty} = 5$ and $R_2 = 1 + 25 + 5 = 31$ are sufficiently large input penalties such that the minimizer of the infinite-horizon control problem \eqref{citoc_lin} is $\mu_\infty^\star(x)=0, \forall x\in\Xcal$. Then, in view of point \ref{item:value}) in Proposition \ref{prop:closed_form}, the value function ${J^\star_\infty}$ induces a partition with infinitely many regions, defined by the hyperplanes
\begin{align}\label{eq:hyperplanes}
\begin{split}
    &c_k x_1 + d_k x_2 = 0.8^k x_1 + k 0.8^{k-1} x_2=0
    \\& \Leftrightarrow 0.8 x_1 + k x_2 = 0, \forall k\in\Z_{\geq1},
\end{split}
\end{align}
as well as the hyperplanes
$x_1=0, \quad x_2=0,$
since the absolute value is non-differentiable wherever its argument is 0.
The hyperplanes described by \eqref{eq:hyperplanes} cross the origin, so they generate infinitely many regions in the compact set $\Xcal$. 
In addition, we see from \eqref{eq:hyperplanes} that the slope of ${J^\star_\infty}$ is different in each of these regions, according to the expression for the coefficients $c_k$ and $d_k$ in \eqref{eq:coeff_convergence} (see more details in Appendix \ref{app:explicit}).
This can also be seen graphically, in Figures \ref{fig:value1} and \ref{fig:value2}, which plot the resulting value function for the given example\footnote{Code available at: \url{https://github.com/fracordi/value-function-pwa}.}.

\begin{figure}
    \centering
    \includegraphics[width=0.7\linewidth]{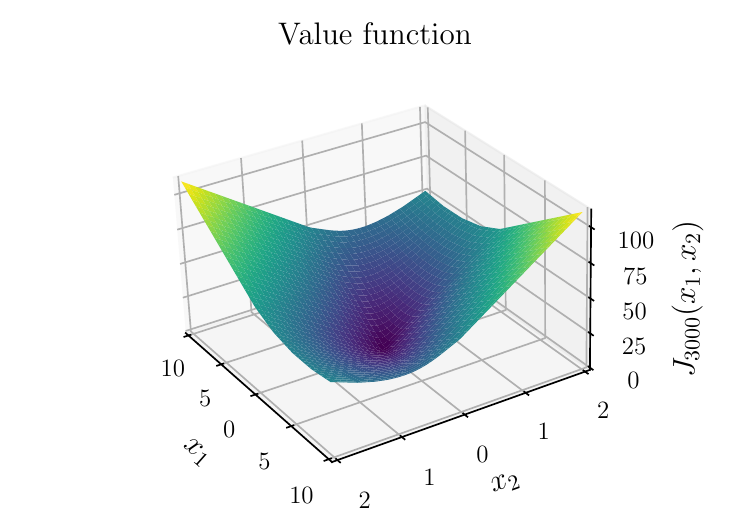}
    \caption{Here, ${J^\star_\infty}$ is approximated with 3000 iterations (thus, $J_{3000}$ is PWA, but the number of pieces grows with the {iteration index}).}
    \label{fig:value1}
\end{figure}
\begin{figure}
    \centering
    \includegraphics[width=.9\linewidth]{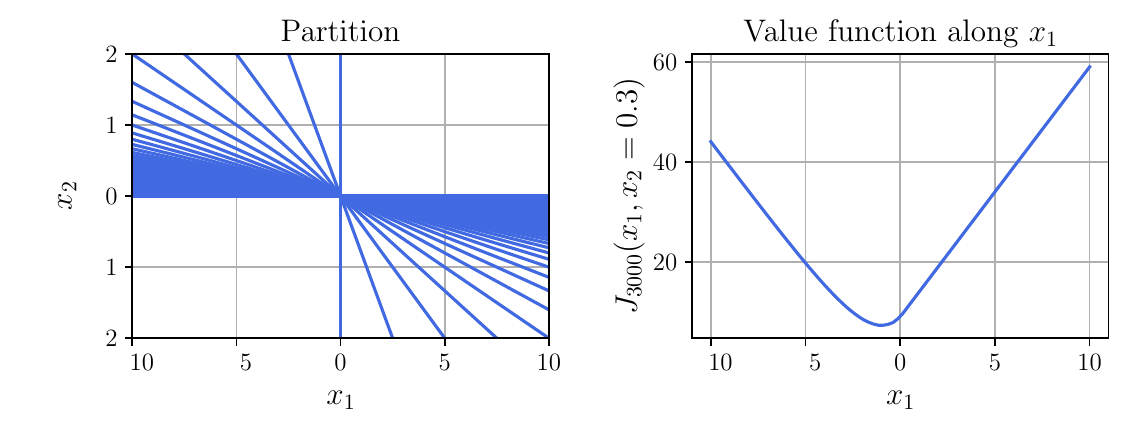}
    \caption{Partition induced by the value function (left), and value function along one coordinate (right).}
    \label{fig:value2}
\end{figure}

\begin{remark}
    {The above analysis gives a more formal interpretation of the results in \cite{baotic2006constrained}: for a finite $k\in\Z_{\geq0}$, it is clear that $J_k$ in \eqref{dp:iter_0}--\eqref{dp:iter_2} is a PWA function of $x$; however, $J^\star_\infty$ is not PWA in the usual sense}. Note that, in our case, this also does not imply that $J^\star_\infty$ is a globally smooth function: indeed, from the example above, it results that $J^\star_\infty$ is a function with infinitely many affine pieces in a compact set.
\end{remark}

\begin{remark}
    {The lack of a proper PWA structure can occur even for} linear dynamical systems with desirable properties, such as stability and controllability. In particular, the latter condition is not sufficient by itself. In \cite{baotic2006constrained} it is mentioned that some dynamical systems have the property of reaching the origin in a finite number of steps, i.e., they are controllable. However, if the input weight is too large in the optimal control problem, the system could still decay exponentially to the origin without reaching it, despite controllability.
\end{remark}

\section{Sufficient condition for PWA structure}\label{sec:suff}
In this section, we give an explicit sufficient condition to ensure that the value function of an infinite-horizon optimal control problem for a PWA system with $\ell_1$ or $\ell_\infty$ stage cost exhibits a finite number of pieces. We first state preliminary assumptions and technical lemmas, then we prove the main result, and finally we conclude with an example. 

\subsection{Preliminary assumptions and lemmas}
In the previous section, we saw that controllability is not sufficient to guarantee this property, as the structure of the value function also depends on the input weight matrix $R$. As we will see in the following proposition, the proposed sufficient condition requires the optimal policy $\mu^\star_\infty$ to steer the system to the origin in one step, for all points in a neighborhood of the origin. In view of the counterexample in the previous section, this will be the case if the matrix $R$ is not too large.
 
The requirements above are now formally stated in the following two assumptions:
\begin{assumption}\label{assum:interior}
There exists a polyhedral region $\mathcal{D}_j$ such that $0_{n+m} \in \mathrm{int}(\mathcal{D}_j)$. Moreover, there exists {a neighborhood $\mathcal{E} \subseteq \mathrm{Proj}_{\mathbb{R}^n}( \mathcal{D}_j)$ of $0_{n}$} such that for every $x \in \mathcal{E}$, there exists an $u$ satisfying $\left[\begin{array}{l}
x \\
u
\end{array}\right] \in \mathcal{D}_j$ and $A_jx+B_ju = 0_n$, where $(A_j, B_j)$ describe the system dynamics in the region $\Dcal_j$.
\end{assumption}
Note that, since $0_{n+m} \in \mathrm{int}(\mathcal{D}_j)${,} the affine term $f_j$ in the PWA dynamics in \eqref{eq:pwa_sys} must be $0$ for region $\Dcal_j$.

\begin{assumption}\label{assum:M}
There is a matrix $M\in\R^{n\times m}$ such that $R = M^\top QB_j$ and $\|M\|_{q}\leq 1$, where $\|\cdot\|_q$ is the dual norm of $\|\cdot\|_p$, i.e., $1/p+1/q = 1$.
\end{assumption}

For the proof of our main result, we need the following two key lemmas:
\begin{lemma}\label{lemma1}
Consider the following parametric optimization problem:
\begin{align}\label{eq:parametric_optimization}
    \min_{u}& \{q(x,u):=\|Ru\|_p+\|Q(A_jx+B_ju)\|_p\}\\
    \text{s.t.}& \left[\begin{array}{l}
x \\
u
\end{array}\right] \in \mathcal{D}_{j}. \nonumber
\end{align}
Under Assumptions \ref{assum:interior} and \ref{assum:M}, there exists a sufficiently small neighborhood $\Omega_r:=\{x\in \mathbb{R}^n| \|x\|_p\leq r\}$ of $0_n$ such that $\forall x \in \Omega_r$, there exists a PWA optimizer $\kappa^\star(\cdot)$ of \eqref{eq:parametric_optimization} satisfying $A_jx+B_j\kappa^\star(x)=0_n$.
\end{lemma}

\emph{Proof:} Since the objective function $q$ is convex and PWA, according to \cite[Section 9.2]{borrelli2017predictive}, there exists a continuous PWA optimizer $\kappa^\star(\cdot)$ and obviously  $\kappa^\star(0_n) = 0_m$.
{From Assumption \ref{assum:interior}, we know that $0_{n+m}\in\text{int}(\Dcal_j)$. This, together with the continuity of $\kappa^\star$, ensures that there exists $r\in\R_{>0}$ and a set $\Omega_r := \{x\in\R^n \mid \|x\|_p\leq r\}$, such that}
\begin{align}\label{eq:omega_condition}
{\left[\begin{array}{l}
x \\
\kappa^\star(x)
\end{array}\right] \in \mathrm{int}(\mathcal{D}_j), \quad \forall x \in \Omega_r.}
\end{align}
In other words, the constraint of {\eqref{eq:parametric_optimization}} is inactive and can be removed when $x \in \Omega_r$. {In addition, if \eqref{eq:omega_condition} is satisfied for a certain $r_1>0$, it is also satisfied for any $r_2\in(0, r_1]$, since in this case we have $\Omega_{r_2}\subseteq \Omega_{r_1}$.
Therefore, it is always possible to select $r$ to be small enough that $\Omega_r\subseteq \Ecal$. Hence, from Assumption \ref{assum:interior}, we know that there exists $u\in\text{Proj}_{\R^m}(\Dcal_j)$ such that $A_j x + B_j u = 0_n$, $\forall x\in\Omega_r$.} 

{Now, we prove that such $u$ is indeed optimal for any $x\in\Omega_r$.} For any $x \in \Omega_r$, define $u^\star = \kappa^\star(x)$, which is one of the optimal solutions. Since the constraint of \eqref{eq:parametric_optimization} can be removed and $q$ is convex, according to \cite{boyd2022subgradients}, the optimality of $u^\star$ is equivalent to 
\begin{equation}\label{eq:subgradient}
    0_m \in \partial q(x,u^\star),
\end{equation}
where $\partial q(x,u^\star)$ is the {subgradient} of $q$ evaluated at $u^\star$. According to the chain rule for {subgradients} \cite{beck2017first}, the right-hand side of \eqref{eq:subgradient} can be expressed as
\begin{align}\label{eq:optimaility}
\begin{split}
    \partial q(x,u^\star) = 
    & R^\top \partial\|\cdot\|_p(Ru^\star) \\& + B^\top_jQ^\top\partial\|\cdot\|_p(Q(A_jx+B_j u^\star)),
\end{split}
\end{align}
where in this case the symbol ``+'' denotes the Minkowski sum of sets, and the subgradient of the $p$-norm function $\|\cdot\|_p(y)$ is computed as $\partial \|\cdot\|_p(y)=\left\{g \mid\|g\|_{q} \leq 1, g^\top y=\|y\|_p\right\}$ \cite{beck2017first, boyd2022subgradients}.

Now, since Assumption \ref{assum:interior} holds, consider any $\bar u$ that satisfies  $\left[\begin{array}{l}
x \\
\bar u
\end{array}\right] \in \mathcal{D}_j$ and $A_jx+B_j\bar u = 0_n$. Considering further any $s \in \partial \|\cdot\|_p(R\bar u )$, define $v = -Ms$, where $M$ is from Assumption \ref{assum:M}. Then, we have the following relation:
\begin{equation}\label{eq:s_and_v}
    R^\top s+B_j^\top Q^\top v=R^\top s-B_j^\top Q^\top M s = 0,
\end{equation}
\begin{equation}\label{eq:v1}
    \|v\|_q = \|Ms\|_q \leq \|s\|_q\leq1,
\end{equation}
where \eqref{eq:s_and_v} holds by definition of $R$ in Assumption \ref{assum:M}, the first inequality in \eqref{eq:v1} follows from $\|M\|_{q} \leq 1$ (in view of the definition of matrix norm), and the second inequality in \eqref{eq:v1} follows from the fact that $s$ is the subgradient of the $p$-norm function. Moreover, since $A_jx+B_j\bar u = 0_n$, obviously we have
\begin{equation}\label{eq:v2}
    v^\top Q(A_jx+B_j\bar u) = \|Q(A_jx+B_j\bar u)\|_p = 0.
\end{equation}
Combining \eqref{eq:v1} and \eqref{eq:v2} yields that $v$ is an element of the subgradient $\partial\|\cdot\|_p(Q(A_jx+B_j \bar u))$. By noticing \eqref{eq:s_and_v} and that \eqref{eq:subgradient} is the sufficient and necessary optimality condition, we conclude that $\bar u$ is one of the optimal solutions to \eqref{eq:parametric_optimization} for the considered $x \in \Omega_r$. $\hfill\Box$

Before stating the second lemma, we note that the following multi-step Bellman equation
\begin{align}\label{eq:bellman}
    J^\star_\infty(x_0) = \min_{\substack{x_1,...,x_{N}, \\ u_0,...,u_{N-1}}} &\sum_{t=0}^{N-1} ( \|Qx_t\|_p+ \|Ru_t\|_p )  +J^\star_\infty(x_N)\\
    \text{s.t.} \ & \eqref{eq:pwa_sys}, \;\forall t\in\{0,1,...,N-1\}\nonumber\\
    &J^\star_\infty(x_N) <\infty \nonumber
\end{align}
holds for any $N \in {\Z_{>0}}$.

\begin{lemma}\label{lemma2}
    Under Assumptions \ref{assum:origin} and \ref{assum:Jfinite}, for any initial state $x_0\in \Xcal_\infty $ and for any $r>0$, there exists a finite positive integer $\bar N(x_0, r)$ such that the optimal predicted state $x^\star_{\bar N(x_0, r)}$ resulting from \eqref{eq:bellman} with $N = \bar N(x_0, r)$ satisfies $x^\star_{\bar N(x_0, r)} \in \Omega_r$.
\end{lemma}
\emph{Proof:} The lemma can be straightforwardly proved by contradiction. Suppose it is not true, then there exists an initial state $x_0\in \Xcal_\infty$ and a positive $r$ such that the optimal solution $x^\star_t$ to problem \eqref{eq:bellman}, {where the cost function is in the limit for $N\to\infty$}, satisfies $\|x^\star_t\|_p>r$. Consequently, we have
\begin{align}\label{eq:J_x0}
    J^\star_\infty(x_0)\geq &\sum_{t=0}^{\infty} \|Qx^\star_t\|_p \geq\sum_{t=0}^{\infty} \lambda(Q) \|x^\star_t\|_p=\infty
\end{align}
where $\lambda(Q)=\frac{\sigma_{\min }(Q)}{\sqrt{n}}$ if $p=1$, ${\lambda(Q)=\frac{\sigma_{\min }(Q)}{\sqrt{n_Q}} \ \text{if} \ p=\infty}$, and $\sigma_{\min }(Q)$ represents the smallest singular value of $Q$, which is positive since $Q$ is full rank. Consequently, \eqref{eq:J_x0} contradicts $x_0\in \Xcal_\infty$. $\hfill\Box$

\subsection{Main result}
We can now state the main result about the structure of the value function, the proof of which uses Lemma \ref{lemma1} and \ref{lemma2}.
\begin{proposition}\label{prop:PWA}
    {Let Assumptions \ref{assum:origin}--\ref{assum:M} hold, and for any $C\in\R_{\geq0}$ define 
     $$ \Xcal_C:= \{x\in\Xcal_\infty: J^\star_\infty(x) \leq C\}.$$
     Then, over any polyhedral subset of $\Xcal_C$, } the optimal value function $J_{\infty}^{\star}$ and the optimal policy $\mu^\star_\infty$ are time-invariant PWA functions 
    {in the sense of Definition \ref{def:pwa}}.
\end{proposition}
\emph{Proof:} The proof consists of two parts.

\emph{Part 1: PWA property of $J^\star_\infty$ over $\Omega_r$.} Consider any $x \in \Omega_r$, by performing the DP update \eqref{dp:iter_1}, we have 
\begin{align*}
    J_1(x) &= \|Qx\|_p,\\
    J_2(x) &= \|Qx\|_p + \|R\kappa^\star(x)\|_p,
\end{align*}
where $\kappa^\star$ is the PWA function defined in Lemma \ref{lemma1}, the first equality holds because of the initialization \eqref{dp:iter_0}, and the last equality is true because of Lemma \ref{lemma1}. 

Then, we will prove by {induction that $J_k = J_2, \forall k\in\Z, k\geq2$. The case for $k=2$ has been shown before via direct inspection. Hence, now we prove that if $J_k(x) = \|Qx\|_p + \|R\kappa^\star(x)\|_p$ then also $J_{k+1}(x) = \|Qx\|_p + \|R\kappa^\star(x)\|_p, \forall k\in\Z, k\geq2$. For the sake of contradiction,} suppose it is not true. 
For any $x \in \Omega_r$, {the DP update \eqref{dp:iter_1} at iteration $k+1$ yields}:
\begin{align}\label{eq:J3}
    {J_{k+1}}(x) = \|Qx\|_p + \min_u \ \{&\|Ru\|_p+ \|Qf_\mathrm{PWA}(x,u)\|_p\nonumber\\
    & + \|R\kappa^\star(f_\mathrm{PWA}(x,u))\|_p\},
\end{align}
since {in \eqref{dp:iter_1} we have $\Xcal_k = \Xcal_0 = \Omega_r$ because $\kappa^\star$ is a feasible policy over $\Omega_r$, which allows us to use the induction hypothesis to expand $J_k(f_\mathrm{PWA}(x,u))$ in \eqref{dp:iter_1}, where} the constraints are removed because $\Omega_r$ is a sufficiently small neighborhood.
Since ${J_{k+1} \ne J_k}$ and $x\in \mathrm{Proj}_{\mathbb{R}^n}( \mathcal{D}_j)$, the minimizer $u'$ of the optimization problem in \eqref{eq:J3} must satisfy $f_\mathrm{PWA}(x,u')=A_jx+B_ju' \ne 0$, {otherwise the resulting cost would coincide with $J_k(x)$}. Now, consider any $u^+$ satisfying $f_\mathrm{PWA}(x,u^+)= 0$. Due to the optimality of $u'$, we have
\begin{align}\label{eq:contradict1}
    \|Ru'\|_p + {J_k}(f_\mathrm{PWA}(x,u'))&<  \|Ru^+\|_p + {J_k}(f_\mathrm{PWA}(x,u^+))\nonumber\\
    &=\|Ru^+\|_p,
\end{align}
where the last equality is the result of ${J_k}(0_n)= 0$. 

On the other hand, since ${J_k}(f_\mathrm{PWA}(x,u))\geq \|Qf_\mathrm{PWA}(x,u)\|_p$ for any $x$ and $u$, we have
\begin{align}\label{eq:contradict2}
    & \|Ru'\|_p+{J_k}(f_\mathrm{PWA}(x,u')) \nonumber
    \\& \geq \min_u\ {\{} \|Ru\|_p+\|Qf_\mathrm{PWA}(x,u)\|_p {\}}
    =\|R\kappa^\star(x)\|_p,
\end{align}
where the {equality in the} second line results from Lemma \ref{lemma1} and $\kappa^\star(x)$ satisfies $f_\mathrm{PWA}(x,\kappa^\star(x))=0$. Since $u^+$ could be any input satisfying $f_\mathrm{PWA}(x,u^+)= 0$, \eqref{eq:contradict2} contradicts \eqref{eq:contradict1}. This proves {that $J_{k+1} = J_k$ and completes the induction step. Therefore, $J^\star_\infty$ is PWA over the set $\Omega_r$.}

\emph{Part 2: PWA property of $J^\star_\infty$ over {$\Xcal_C$}.} Now, consider any {$x_0 \in \Xcal_C$, for any $C\in\R_{\geq0}$.} According to Lemma \ref{lemma2}, $J^\star_\infty(x_0)$ can be represented by \eqref{eq:bellman} with $N= \bar N(x_0,r)$. From Lemma \ref{lemma2}, the optimal predicted state $x^\star_{\bar N(x_0, r)}$ of \eqref{eq:bellman} satisfies $x^\star_{\bar N(x_0, r)} \in \Omega_r$. {For a given horizon $N$ and initial condition $x_0$, let us denote by $x^\star_N(x_0)$ the $N$-th predicted state when $N$ is used in \eqref{eq:bellman}, from the initial state $x_0$. It is clear that if $x^\star_{\bar N(x_0, r)} \in \Omega_r$, then $x^\star_{N}(x_0) \in \Omega_r, \forall N\geq \bar N(x_0, r)$, since the next state is $0_n$ under the optimal policy in $\Omega_r$. Also, without loss of generality, let us assume in this proof that $\bar N(x_0, r)$ is the first time step for which $x^\star_{\bar N(x_0, r)} \in \Omega_r$, and that $\bar N(x_0, r) = 0$ if $x_0\in\Omega_r$.
}

{We now show that there exists a uniform upper bound for $\bar N(x_0,r)$ over $\Xcal_C$. Let us consider an arbitrary $\eta_r\in\R$ such that $0 < \eta_r \leq \inf_{x\not\in\Omega_r}\|Qx\|_p$, with $p\in\{1,\infty\}$. Note that $\inf_{x\not\in\Omega_r}\|Qx\|_p > 0$ since $0_n\in\Omega_r$ and $\|Qx\|_p=0$ if and only if $x=0_n$, in view of Assumption \ref{assum:Jfinite}. Also, let $u^\star_0(x_0),...,u^\star_{\bar N(x_0, r)}(x_0)$ be the optimal input sequence resulting from \eqref{eq:bellman}. Following the same lines of the proof of \cite[Theorem 1]{chmielewski1996constrained}, we see that 
\begin{align*}
    & J^\star_\infty(x_0) 
    \\& = \sum_{t=0}^{\bar N(x_0, r)-1} \left( \|Qx_t\|_p + \|R u^\star_t(x_0)\|_p\right) +J^\star_\infty(x_{\bar N(x_0, r)})
    \\& \geq \sum_{t=0}^{\bar N(x_0, r)-1} \|Qx_t\|_p 
    \\& \geq \eta_r \bar N(x_0, r).
\end{align*}
Together with the definition of $\Xcal_C$, we then obtain
$$\sup_{x_0\in\Xcal_C} \bar N(x_0, r) \leq \sup_{x_0\in\Xcal_C} \frac{J^\star_\infty(x_0)}{\eta_r} \leq \frac{C}{\eta_r}.$$
Hence, there exists $\bar N(r):=\sup_{x_0\in\Xcal_C} \bar N(x_0, r)<\infty$ such that $\bar N(x_0, r) \leq \bar N(r), \forall x_0 \in\Xcal_C$.} Since $J^\star_\infty$ is PWA over the set $\Omega_r$ {and since $\Omega_r$ is reached in at most $\bar N(r)<\infty$ steps for any $x_0\in\Xcal_C$}, there exists an optimal policy that is PWA over {any polyhedral subset of $\Xcal_C$}, and the optimal value function $J^\star_\infty$ is also PWA {with a finite number of affine pieces over such set}, according to \cite[Theorem 3]{borrelli2005dynamic} and \cite[Theorem 1]{he2024approximatea_pwa}. This completes the proof of Proposition \ref{prop:PWA}. $\hfill\Box$

\subsection{Example}
To test the validity of our results\footnote{Code available at: \url{https://github.com/fracordi/value-function-pwa}.}, we consider $Q = B = I_2$, $a_1=0.8, a_2=1$; i.e., as in Section \ref{sec:counterexample}. We consider the state space $\Xcal=[-0.25, 0.25]^2$. We choose $M$ in Assumption \ref{assum:M} from the set $\{0.9I_2, I_2, 1.1I_2\}$. Figure \ref{fig:region} shows the resulting $\Omega_r$ for each choice of $M$. The regions have been determined by repeatedly solving the optimization problem \eqref{eq:parametric_optimization} for $x\in\Xcal$ (via gridding), and by tracing the points for which the next state is 0. More specifically, given that $M=\kappa I_2$, it results that any value of $\kappa\in[0,1)$ gives the light-blue region in Figure \ref{fig:region} (i.e., the box $[-0.25, 0.25]^2$), $\kappa=1$ gives the dark-blue region, and $\kappa>1$ yields an empty set. This confirms the validity of Proposition \ref{prop:PWA}.
\begin{figure}
    \centering
    \includegraphics[width=.8\linewidth]{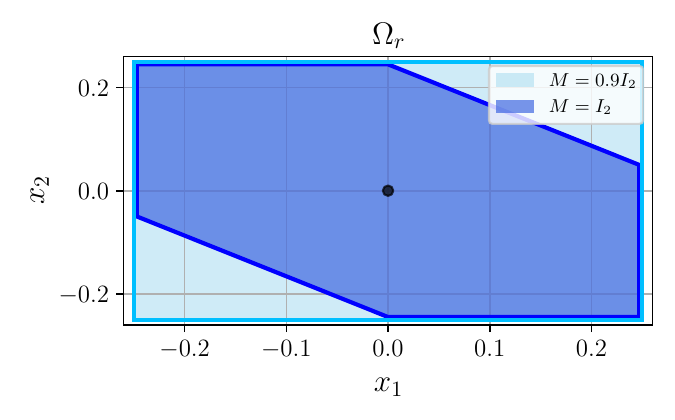}
    \caption{$\Omega_r$ for $M=0.9I_2$ (lightblue), and for $M=I_2$ (blue). For $M=1.1I_2$ $\Omega_r$ is an empty set.}
    \label{fig:region}
\end{figure}

At this point, one could wonder whether the condition on $M$ in Assumption \ref{assum:M} is both necessary and sufficient. It turns out that such condition is necessary and sufficient for the DP iterations \eqref{dp:iter_0}--\eqref{dp:iter_2} to terminate in two iterations; however, it is not necessary to just guarantee that the value function has a finite number of pieces.
For example, it is possible to design $R$ such that Assumption \ref{assum:M} is not satisfied, but the DP iterations \eqref{dp:iter_0}--\eqref{dp:iter_2} terminate in a finite number of steps (perhaps larger than two), in a set close to the origin; hence, the value function still has a finite number of pieces. 

Consider again our numerical example with $A_j = A$ defined as before, $\Xcal$ is now $[-0.25, 0.25]\times[-0.05, 0.05]$, and $M=1.1I_2$ is designed to violate Assumption \ref{assum:M}.  Note that this implies $R=\text{diag}(1.1, 1.1)$. From the DP iterations \eqref{dp:iter_0}--\eqref{dp:iter_2}, we have, $\forall x\in\Xcal$
\begin{align}\label{eq:dp_example}
\begin{split}
    & J_0(x) = 0
    \\& J_1(x) = \|x\|_1
    \\& J_2(x) = \|x\|_1 + \|Ax\|_1,
\end{split}
\end{align}
where $\mu_1(x) = \mu_2(x) = 0$, $\forall x\in\Xcal$. That fact that $\mu_1 = 0$ is immediate. Concerning $\mu_2$, a direct application of Lemma \ref{lemma:aux} in Appendix \ref{app:proof} gives $\mu_2(x) = 0, \forall x\in\Xcal$.

Then, for $\mu_3$ we have
\begin{align*}
    \mu_3(x) &= \begin{aligned}[t]
    \arg\min_{u} \ & \|x\|_1 + \|Ru\|_1 + \|A x + Bu\|_1
    \\& + \|A(A x + Bu)\|_1,
    \end{aligned}
\end{align*}
and, via numerical simulations, it can be seen that $ u^\star = -Ax$. By plugging this expression into the cost function, we find 
\begin{align*}
    J_3(x) = \|x\|_1 + \|RAx\|_1,
\end{align*}
and in a similar way we obtain $\mu_4 = \mu_3$; hence, $J_4 = J_3$.
This means that three iterations are sufficient for the value function to converge, for all $x\in\Xcal$. This implies ${J^\star_\infty}(x) = J_3(x)\ \forall x\in\Xcal$, thus ${J^\star_\infty}$ is PWA. 

We conclude that, although Proposition \ref{prop:PWA} provides general sufficient conditions for the value function to converge in exactly two iterations, problem-specific conditions can be more relaxed for the value function to be PWA.

\section{Conclusions}
\label{sec:conclusions} 
This technical note has presented a detailed study of the value function in CITOC problems for PWA systems  with $\ell_1$ or $\ell_\infty$ stage cost. {By means of an explicit example, we have clarified some existing results in the literature}, and thereafter we have proposed sufficient conditions {that ensure a proper PWA structure, i.e., with a finite number of pieces over compact sets}. Future work aims {to extend} these results to stochastic PWA systems.

\section*{Appendix}

\subsection{Proof of Proposition \ref{prop:closed_form}}\label{app:proof}
First, we need a preliminary lemma:
\begin{lemma}\label{lemma:aux}
    Consider any functions $F_i:\Xcal^2\rightarrow\R$, with $i\in\{1,...,h\}, h\in\Z_{\geq1}$, and $\alpha_{1,i}, \alpha_{2,i} \in \R$. Then if $R_1 \geq \sum_{i=1}^h |\alpha_{1,i}|$ and $R_2 \geq \sum_{i=1}^h |\alpha_{2,i}|$, we have that
    \begin{align}
         & \min_{u_1, u_2} R_1|u_1| + R_2|u_2| + \sum_{i=1}^h|F_i(x_1, x_2) + \alpha_{1,i} u_1 + \alpha_{2,i} u_2| \nonumber
         \\& = \sum_{i=1}^h|F_i(x_1, x_2)| \label{eq:triangle}
    \end{align}
holds $\forall x_1, x_2 \in \Xcal$, and the minimum is achieved for $u_1 = u_2 = 0$.
\end{lemma}
\emph{Proof:}
First, the triangle inequality yields, $\forall i\in\{1,...,h\}, \forall x_1, x_2\in\Xcal, \forall u_1, u_2\in\R$
\begin{align*}
    &|F_i(x_1, x_2)| 
    \\& = |- (\alpha_{1,i} u_1 + \alpha_{2, i} u_2) + F_i(x_1, x_2) + \alpha_{1,i} u_1 + \alpha_{2, i} u_2| 
    \\& \leq |\alpha_{1, i}| |u_1 | + |\alpha_{2, i}||u_2| + |F_i(x_1, x_2) + \alpha_{1,i} u_1 + \alpha_{2, i} u_2|.
\end{align*}
Summing over $i\in\{1,...,h\}$, we obtain
\begin{align*}
    \sum_{i=1}^h |F_i(x_1, x_2)| 
    &\leq \sum_{i=1}^h \Big(|\alpha_{1, i}| |u_1 | + |\alpha_{2, i}||u_2| 
    \\& \qquad \quad + |F_i(x_1, x_2) + \alpha_{1,i} u_1 + \alpha_{2, i} u_2| \Big)
    \\&\leq R_1 |u_1| + R_2|u_2| 
    \\& \qquad \quad + \sum_{i=1}^h |F_i(x_1, x_2) + \alpha_{1,i} u_1 + \alpha_{2, i} u_2| ,
\end{align*}
where the second inequality holds by the assumption on $R_1, R_2$. Also, note that equality is attained for $u_1=u_2 = 0$. Since the inequality above holds for all $u_1, u_2 \in\R$, we conclude that $u_1 = u_2 = 0$ achieves the minimum in \eqref{eq:triangle}, for any $x_1, x_2\in\Xcal$. $\hfill\Box$

Now we can prove Proposition \ref{prop:closed_form}.
\begin{itemize}
    \item Proof of \ref{item:r_finite}). The quantities $c_k$ and $d_k$ evolve according to a linear dynamical system (cf. \eqref{eq:coeff_updates}), whose matrix has two eigenvalues equal to $a_1$, which belongs to $(-1, 1)$ by assumption. Hence, this system is exponentially stable. Observe from \eqref{eq:coeff_updates} that we have
    \begin{align}\label{eq:r_closed_form}
    \begin{split}
        & r_{1,k} = r_{1,0} + \sum_{l=0}^{k-1} |c_l|, 
        \\& r_{2,k} = r_{2,0} + \sum_{l=0}^{k-1} (|d_l| + |c_l|).
    \end{split}
    \end{align}
    
    \noindent 
    Since $c_k$ and $d_k$ decay exponentially to 0, $r_{i,\infty}$ are finite, $\forall i\in\{1,2\}$.

    \item Proof of \ref{item:value}). We proceed by induction, over $k\in\Z_{\geq1}$. The result is true for $k=1$, since we have $J_1(x) = |x_1| + |x_2|$.
    Then assume 
    $J_{k}(x) = \sum_{j=0}^{k-1} (|c_j x_1 + d_j x_2| + |c_j x_2|), \forall x\in\Xcal, \forall k\in\Z_{\geq 1}$
    is true. We have to prove that the relation also holds at iteration $k+1$. From the DP step \eqref{dp:iter_1}, we have
{
    \begin{align*}
        &J_{k+1}(x) 
        \\& \begin{aligned}[t]
            = \min_{u} \big\{ & |x_1| + |x_2| + R_1|u_1| + R_2|u_2| + J_k(Ax + Bu),
                    \\& \text{s.t.}\ Ax + Bu \in \Xcal\big\}.
        \end{aligned}
    \end{align*}
}

{Note that this one-step-ahead problem is feasible, since we assume $\Xcal$ to be invariant for the autonomous system, which implies that $u=0$ is a feasible solution. Now, since $Ax+Bu$ is constrained to be in $\Xcal$, we can leverage the induction hypothesis to expand the expression $J_k(Ax + Bu)$, which, according to the system dynamics \eqref{eq:lin_sys}, yields:}
\begin{align}
    & J_{k+1}(x) = |x_1| + |x_2| + \min_{u_1, u_2} \Big\{
    R_1|u_1| + R_2|u_2|  \nonumber
    \\& \quad + \sum_{j=0}^{k-1} |c_j (a_1 x_1 + a_2 x_2 + u_1) + d_j (a_1 x_2 + u_2)| \nonumber
    \\& \quad + \sum_{j=0}^{k-1} |c_j (a_1 x_2 + u_2)|, \ {\text{s.t.}\ Ax + Bu \in\Xcal} \Big\}. \label{eq:minimization}
\end{align}

    \noindent Now, {let us disregard the constraint $Ax + Bu \in\Xcal$, which will be checked a posteriori.} Note that $R_1 \geq r_{1,\infty} \geq \left(1 + \sum_{l=0}^{k} |c_l|\right)$ and $R_2 \geq r_{2,\infty} \geq \left(1 + \sum_{l=0}^{k} (|d_l| + |c_l|)\right)$ hold $\forall k\in\Z_{\geq0}$, where, in both cases, the first inequality is by the assumption on $R_i$ in Proposition \ref{prop:closed_form} item 2), and the second one by definition of $r_{i,\infty}, \forall i\in\{1,2\}$, and \eqref{eq:r_closed_form}. Then, we can apply Lemma \ref{lemma:aux} to the minimization {of the objective function in \eqref{eq:minimization}, since it} is of the type \eqref{eq:triangle} with $h=2k$, and, for $i=1,...,k$, set $F_i(x_1,x_2) = a_1 c_{i-1}x_1 + (a_2c_{i-1}+ a_1d_{i-1})x_2, {\alpha_{1,i}}=c_{i-1}, {\alpha_{2,i}}=d_{i-1}$; for $i=k+1,...,2k$, set $F_i(x_1,x_2) = a_1 c_{{i}-k-1}x_2, {\alpha_{1,i}}=0, {\alpha_{2,i}}=c_{i-k-1}$.
    Lemma \ref{lemma:aux} implies that $u_1=u_2=0$ is the minimizer in \eqref{eq:minimization}{; moreover, the state constraint $Ax + Bu = Ax \in \Xcal$ is satisfied since $\Xcal$ is assumed to be an invariant set for the autonomous system. Hence,} in view of the updates \eqref{eq:coeff_updates}, we have
    \begin{align*}
        &J_{k+1}(x) 
        \\& =  |x_1| + |x_2| 
        \\ & \quad + \sum_{j=0}^{k-1} \left(|c_j (a_1 x_1 + a_2 x_2) + d_j a_1 x_2| + |c_j a_1 x_2 |\right)
        \\& = {\sum_{j=0}^{k} \left(|c_j x_1 + d_j x_2| + |c_j x_2 |\right)}.
    \end{align*}
This completes the inductive step. $\hfill\Box$

\end{itemize}

\subsection{Explicit form of the value function}\label{app:explicit}

In this section, we give an explicit expression for the value function of the counterexample in Section \ref{sec:counterexample} to show analytically that it is not PWA. For simplicity, we consider $a_1 = a_2 = a>0$ and that $J_k (x) = P_k (x) + Q_k (x)$, where $P_k (x) = \sum_{j=0}^{k-1} a^j |x_1+ jx_2|$ and $Q_k (x) = \sum_{j=0}^{k-1} a^j |x_2|$, $\forall k\in\Z_{\geq1}$. 
Then, considering the region defined by $x_1 \geq 0$ and $x_2 <0$, we notice that 
$x_1+ jx_2 \geq 0 \Leftrightarrow j \leq \lfloor -\frac{x_1}{x_2} \rfloor$;
hence, by defining $h:=\min(k-1,  \lfloor -\frac{x_1}{x_2} \rfloor)$, we can write
\begin{align*}
    P_k(x) &= \sum_{j=0}^{h}a^j (x_1+ jx_2)+\sum_{j=h+1}^{k-1}a^j (-x_1- jx_2),
\end{align*}
where the second summation is empty whenever $h+1>k-1$. Then, we have:
\begin{align*}
&P_k(x)
\\& =2 \sum_{j=0}^{h}a^j (x_1+ jx_2) - \sum_{j=0}^{k-1}a^j (x_1+ jx_2)\\
    & = \frac{1-a^{h+1}}{1-a}\cdot 2x_1 + \frac{a-a^{h+2}+({h+1})a^{h+1}(a-1)}{(1-a)^2}\cdot 2 x_2\\
    & \;\;\;\;- \frac{1-a^k}{1-a}x_1 - \frac{a-a^{k+1}+ka^k(a-1)}{(1-a)^2}x_2,
\end{align*}

\noindent where we used known identities for geometric series (cf. \cite{apostol1991calculus}, Section 10.8). Note that $P_k$ has a finite number of pieces for any finite $k\in\Z_{\geq1}$, since $h$ and $k$ are always finite in this case.
Letting $k \to \infty$ yields
\begin{align*}
& P_\infty(x) 
\\& = \frac{1-2a^{h+1}}{1-a}x_1 + \frac{a-2a^{h+2}+2({h+1})a^{h+1}(a-1)}{(1-a)^2} x_2\\
& = \frac{1-2a^{h+1}}{1-a}x_1+\frac{a+2ha^{h+2}-2ha^{h+1}-2a^{h+1}}{1-a^2}x_2,\\
& \quad\quad \text{if}\; -hx_2\leq x_1<(-h+1)x_2,\;h=0,1,...\ .
\end{align*}

\noindent Note that for any $x$ such that $\frac{x_1}{x_2}\to -\infty$ (e.g., for $-x_2$ significantly smaller than $x_1$), we have that $h\to\infty$. 
Hence, the last equality indicates that $P_\infty$ is an affine function of $x$ over the polyhedra defined by $\{x\in \mathbb{R}^2 | -h x_2\leq x_1<(-h+1)x_2, \;x_1 \geq 0,\; \text{and}\; x_2 \leq0\}, h=0,1, ...$, with a different slope in each polyhedron, and the number of affine pieces grows unbounded as $\frac{x_1}{x_2}\to -\infty$.

\bibliographystyle{IEEEtran}
\bibliography{bibliography}

\end{document}